\pdfoutput=1
\documentclass[12pt]{article}

\usepackage[utf8]{inputenc}

\usepackage{geometry}
\usepackage{color}
\usepackage{amsmath}
\usepackage{amssymb}
\usepackage{amsthm}
\usepackage{graphicx}
\usepackage{hyperref}
\usepackage[font=scriptsize]{caption} 
\usepackage{float}
\usepackage{mathrsfs}
\usepackage{subcaption}
\usepackage{booktabs}
\usepackage{makecell}
\usepackage[title]{appendix}

\newtheorem{Assumption}{Assumption}
\newtheorem{Lemma}{Lemma}
\newtheorem{Theorem}{Theorem}

\newtheorem{Example}{Example}

\newlength{\wdth}

\def\phi{\varphi}

\title{Estimation and Application of the Convergence Bounds for Nonlinear Markov Chains}

\author{Kaichen Xu \footnote{Zhongnan University of Economics and Law, Wuhan, China; email: Kaichenxu358@gmail.com}}

\begin{document}

\maketitle

\begin{abstract}
Nonlinear Markov Chains (nMC) are regarded as the original (linear) Markov Chains with nonlinear small perturbations. It fits real-world
data better, but its associated properties are difficult to describe. A new approach is proposed to analyze the ergodicity and even 
estimate the convergence bounds of nMC, which is more precise than existing results. In the new method, Coupling Markov about homogeneous
Markov chains is applied to reconstitute the relationship between distribution at any times and the limiting distribution. The convergence
bounds can be provided by the transition probability matrix of Coupling Markov. Moreover, a new volatility called TV Volatility can be
calculated through the convergence bounds, wavelet analysis and Gaussian HMM. It's tested to estimate the volatility of two securities 
(TSLA and AMC). The results show TV Volatility can reflect the magnitude of the change of square returns in a period wonderfully.

\medskip

\noindent
{\em Keywords:} Nonlinear Markov Chains; Convergence Bounds; Coupling Markov; Volatility of Securities

\medskip

\noindent
MSC 2010: 60J10, 60J20, 60J99

\end{abstract}

\section{Introduction}\label{sec1}
In recent years, Markov chains have become one of the most prominent stochastic processes with the wide range of applications, such as 
image generation \cite{Ho,Song}, Natural Language Processing \cite{Liu} and Reinforcement Learning \cite{Foerster,Yu}. 

However, the assumptions in the basic Markov chains aren't always suitable for real-world situations. Thus, nonlinear Markov chains are 
defined for relaxing assumptions. Considering a discrete-time irreducible p-dimension Markov chain $(X_n)_{n \in \mathbb{Z}_+}$ with the 
distribution law $\mu_n := Law(X_n)=(\mu_n^1,\mu_n^2,\cdots,\mu_n^p)$ and transition probability matrix $\mathbb{P}_{p \times p}$, it 
always exists $\mu_n^T = \mu_{n-1}^T \mathbb{P}$. Actually, it can be also written as simultaneous equations:
\begin{equation}\label{def1}
    \begin{cases}
        P(1,1)\mu_{n-1}^1+P(2,1)\mu_{n-1}^2+\cdots+P(p,1)\mu_{n-1}^p = \mu_n^1 \\
        P(1,2)\mu_{n-1}^1+P(2,2)\mu_{n-1}^2+\cdots+P(p,2)\mu_{n-1}^p = \mu_n^2 \\
        \hspace{140pt} \vdots \\
        P(1,p)\mu_{n-1}^1+P(2,p)\mu_{n-1}^2+\cdots+P(p,p)\mu_{n-1}^p = \mu_n^p
     \end{cases}
\end{equation}
In the equations (\ref{def1}), $\mu_n^i$ can be {\em linear} expressed by $\mu_{n-1}$. When $\mu_n^i$ can't be linear expressed and 
it's influenced by high order terms of $\mu_n^1,\mu_n^2,\cdots,\mu_n^p$, Markov chain will be {\em nonlinear} Markov chain (nMC)
\footnote{The more rigorous definition including Markov processes on general state spaces can be found in \cite{Del,Kolokoltsov}} .
It also exists (\ref{def2}),
\begin{equation}\label{def2}
    \begin{cases}
        \sum_{i=1}^pP(i,1)\mu_{n-1}^i+\sum_{i=1}^pP_2(i,1)(\mu_{n-1}^i)^2+\cdots=\mu_n^1 \\
        \sum_{i=1}^pP(i,2)\mu_{n-1}^i+\sum_{i=1}^pP_2(i,2)(\mu_{n-1}^i)^2+\cdots=\mu_n^2 \\
        \hspace{140pt} \vdots \\
        \sum_{i=1}^pP(i,p)\mu_{n-1}^i+\sum_{i=1}^pP_2(i,p)(\mu_{n-1}^i)^2+\cdots=\mu_n^p
     \end{cases}
\end{equation}
For simple expression, let$P(x,y)=P_1(x,y)$, and $P_{\mu_{n-1}}(x,y)=\sum_{j=0}P_{j+1}(x,y)(\mu_{n-1}^x)^j$, then 
$\mu_n^T = \mu_{n-1}^T \mathbb{P}_{\mu_{n-1}}$. The idea of small nonlinear perturbations of linear Markov chains makes the randomness 
of the system greatly increased, which means nMC can fit the factual situation more appropriately but also difficult to calculating the 
result. 

Ergodicity, and the limiting distribution of nMC are one of the most important properties. It has been studied in recent research 
\cite{Butkovsky1,Butkovsky2,Butkovsky3,Shchegolev1,Shchegolev2,Shchegolev3}. 
Most of research tries to find a suitable method which transforms nMC to linear MC. But nMC's distribution doesn't always converge to 
a special value, which may be different from usual MC
\footnote{Actually, MC sometimes have the complex convergence condition, which may be also suitable for nMC. Veretennikov provides more 
detailed literature review in \cite{Veretennikov1,Veretennikov2}} . 
So, it always provides convergence bounds as a weakened condition. 

In Butkovsky's research \cite{Butkovsky2} it has been demonstrated that the Markov-Dobrushin (M-D) condition taken from linear MC is 
also applicable for nMC. Let a nMC $(X_n)_{n \in \mathbb{Z}_+}$ have the measurable state space $\mathscr{S}(E,\mathscr{E})$, and 
$\mathscr{P}(E)$ is a probability measure set defined on $\mathscr{S}$. The transition probability can be expressed as 
$P_\mu(x,B)=P(X_n \in B|X_{n-1}=x)$, where $x \in E, B \in \mathscr{E}, n \in \mathbb{Z}_+$. 
M-D condition for convergence is shown as (\ref{M-D}),
\begin{equation}\label{M-D}
    \sup_{\mu,\nu\in\mathscr{P}(E)} \|P_\mu(x,\cdot)-P_\nu(x^{\prime},\cdot) \|_{TV} \leq 2(1-\alpha),\quad \alpha>0,
    \quad x,x^{\prime} \in E
\end{equation}
Here $\|\mu-\nu\|_{TV}=\sup_{A \in \mathscr{E}} |\mu(A)-\nu(A)| = \int_E |\mu(dx)-\nu(dx)|$, which has another name for total variation 
distance between $\mu$ and $\nu$ in \cite{Rudin}. It's worth noting that (\ref{M-D}) doesn't always guarantee the convergence of nMC. 
And $\alpha$ is a Markov-Dobrushin's coefficient for nMC, which can be calculated as (\ref{alpha}),
\begin{equation}\label{alpha}
    \alpha := \inf_{x,x^{\prime},\mu,\nu} \int_E (\frac{P_\mu(x,dy)}{P_\nu(x^{\prime},dy)} \wedge 1) P_\nu(x^{\prime},dy)
\end{equation}
Specially, when the the state space is finite or countable, $\alpha$ is as follows,
\begin{equation*}
    \alpha := \inf_{x,x^{\prime},\mu,\nu} \sum_{y} P_\mu(x,y) \wedge P_\nu(x^{\prime},y)
\end{equation*}
In \cite{Butkovsky2}, there is another important condition (\ref{Lipschitz}) called as Lipschitz condition.
\begin{equation}\label{Lipschitz}
    \|P_\mu(x,\cdot)-P_\nu(x,\cdot) \|_{TV} \leq \lambda \|\mu - \nu \|_{TV},\quad \lambda \in [0,\alpha],\quad x \in E,
    \quad \mu,\nu \in \mathscr{P}(E)
\end{equation}
The result of condition (\ref{M-D}) and (\ref{Lipschitz}) reads further,
\begin{equation}\label{M-D-L}
    \sup_{\mu_n,\nu_n \in \mathscr{P}(E)} \|\mu_n - \nu_n\|_{TV} \leq 2(1-\alpha+\lambda)^n,\quad n \geq 1
\end{equation}
Additionally, when $\lambda < \alpha$ in (\ref{M-D-L}), $2(1-\alpha+\lambda)^n$ will be converged to 0 at exponential speed, which means
nMC exists uniform exponential ergodicity. But in case $\lambda = \alpha$, the condition (\ref{M-D-L}) should be defined as a linear 
convergence at $1/n$ speed in \cite{Butkovsky3,Shchegolev1}.
\begin{equation*}
    \sup_{\mu_n,\nu_n \in \mathscr{P}(E)} \|\mu_n - \nu_n\|_{TV} \leq \frac{2}{\lambda n},\quad n \geq 1
\end{equation*}
When $\lambda > \alpha$, convergence bounds $2(1-\alpha+\lambda)^n$ increase exponentially as n increase uniformly. But the TV distance 
between $\mu_n$ and $\nu_n$ has a maximum 2. It's a complicate situation, and \cite{Butkovsky1,Butkovsky3,Shchegolev2} gave further
research. Although Butkovsky has estimated the convergence bounds, it still can be enhanced for more precise result. 

Moreover, Shchegolev has demonstrated when $P_\mu(x,\cdot)$ and $P_\nu(x,\cdot)$ are substituted by k-step transition probability, 
the condition (\ref{M-D})-(\ref{M-D-L}) will still exist in \cite{Shchegolev1,Shchegolev2}. The result from \cite{Shchegolev2} 
is based on \cite{Shchegolev1}, but it's more explict. It expands the Lipschitz condition (\ref{Lipschitz}) to k-step transition.
\begin{equation*}
    \|P_\mu^{(k)}(x,\cdot)-P_\nu^{(k)}(x,\cdot) \|_{TV} \leq \lambda_k \|\mu - \nu \|_{TV},\quad \lambda \in [0,\alpha],
    \quad x \in E,\quad \mu,\nu \in \mathscr{P}(E)
\end{equation*}
Let the limiting distribution $\pi = \nu$, and Shchegolev has proved the convergence bounds (\ref{k-step}) is true for any probability 
measure $\mu\in\mathscr{P}(E)$.
\begin{equation}\label{k-step}
    \|\mu_n-\pi\|_{TV} \leq \|\mu_0 - \pi\|_{TV}(1-\alpha_k+\lambda_k)^{[n/k]}(1+\lambda_1)^{n\;{\rm mod}\;k}
\end{equation}
Specially, in the case $\alpha_k=\lambda_k$, a better expression is as follows.
\begin{equation*}
    \|\mu_n-\pi\|_{TV} \leq \frac{\|\mu_0 - \pi\|_{TV}}{2+\lambda_kn\|\mu_0 - \pi\|_{TV}}(1+\lambda_1)^{n\;{\rm mod}\;k}
\end{equation*}

There are many others exploring the convergence bounds or ergodicity of nMC, and some of them even apply nMC to complex fact situation. 
In \cite{Saburov}, Saburov proposes a polynomial stochastic operator $\mathfrak{P}:\mathscr{S}^{p-1}\rightarrow\mathscr{S}^{p-1}$, which 
can describe the nonlinear relationship of nMC distribution $\mu_n \in \mathscr{S}^{p-1}$ in equations (\ref{def2}). It also expands M-D 
condition (\ref{M-D}) and provides second-order, third-order coefficient of (\ref{alpha}) for strong ergodicity. In \cite{Zhang}, Zhang 
pays attention to complex epidemic system, and uses stochastic nMC to explore ergodic stationary distribution of the SIRI system. It 
judges convergence by probability, instead of the TV distance. Coincidentally, in \cite{Aghajani}, it solves some problems from 
stochastic system so-called GI/GI/N queueing network by nMC. 

However, there are still some problems in the research. First of all, the method to estimate convergence bounds isn't precise enough. 
Although some of them can even realize stronger convergence, there are many extra assumptions that most nMC can't respect to. Second, 
many studies are just based solely on mathematics. After estimating convergence bounds, it needs to test nMC by solving some practical 
problems, which even means nMC should bet MC or other classic models. 

These problems will be further analyzed in this paper. In section \ref{sec2}, some assumptions and lemmas will be introduced for the 
proof of main result, especially Coupling Markov theory. Actually, the assumptions and lemmas use the similar ideas with appropriate 
changes taken from recent research \cite{Shchegolev3,Veretennikov2}. In section \ref{sec3}, it will enhance convergence condition further 
by the spectral radius, optimization and etcetera. The work may be associated with \cite{Leadbetter}, but it was changed moderately for 
nMC. Moreover, it will use some numerical examples to test and explain the resluts, and will compare the new estimation with classic M-D 
condition (\ref{M-D}) and fact result. In section \ref{sec4}, the main result will be applied in quantitative finance. Considering that 
TV distance describes the difference between current distribution $\mu_n$ and limiting distribution $\pi$, it can measure the volatility 
of system. If the TV distance becomes bigger, the difference will be bigger and the current system will be harder to reach a steady 
limiting system. According to the simple idea, a new indicator named TV Volatility can be defined to measure the risk of the returns in a
period. Importantly, it overcomes short-sightedness of the volatility through GARCH model, and even can acutely catch the changes of 
series. At least, in section \ref{sec5}, the conclusion will summerize all the results in the paper, and provides some ideas for further 
research.

\section{Assumption and Lemma}\label{sec2}
nMC in equations (\ref{def2}) can be regarded as the linear MC with small nonlinear perturbations. So, two assumptions are proposed for 
transforming nMC to linear MC. 

\begin{Assumption}
    It exists a homogeneous linear MC with nMC. The nMC has transition probability $P_{\mu_i}$, and linear MC has transition probability 
    $P^*$ which is irrelevant to distribution. For any step $n$, each state $x,y$, it always exists, 
    \begin{equation}\label{assumption1}
        \frac{P^*(x,y)}{P_{\mu_n}(x,y)} \leq (1+\gamma),\quad \forall n \in \mathbb{Z}_+,\quad x,y \in E
    \end{equation}
    where $\gamma$ is a small quantity. 
\end{Assumption}

The assumption (\ref{assumption1}) defines a homogeneous linear MC,which is based on original nMC but denoise stochastic small 
perturbations. By the way, there are maybe more than one linear MC satisfing assumption. 

\begin{Assumption}
    The nMC and linear MC mapped by nMC have the same observed trajectory. When it obtained a factual trajectory, the difference 
    between nMC and MC only is the distribution of stochastic states and probability to get the trajectory. 
\end{Assumption}

\begin{figure}[htb]
    \centering
    \begin{minipage}[t]{0.48\textwidth}
    \centering
    \includegraphics[width=\linewidth]{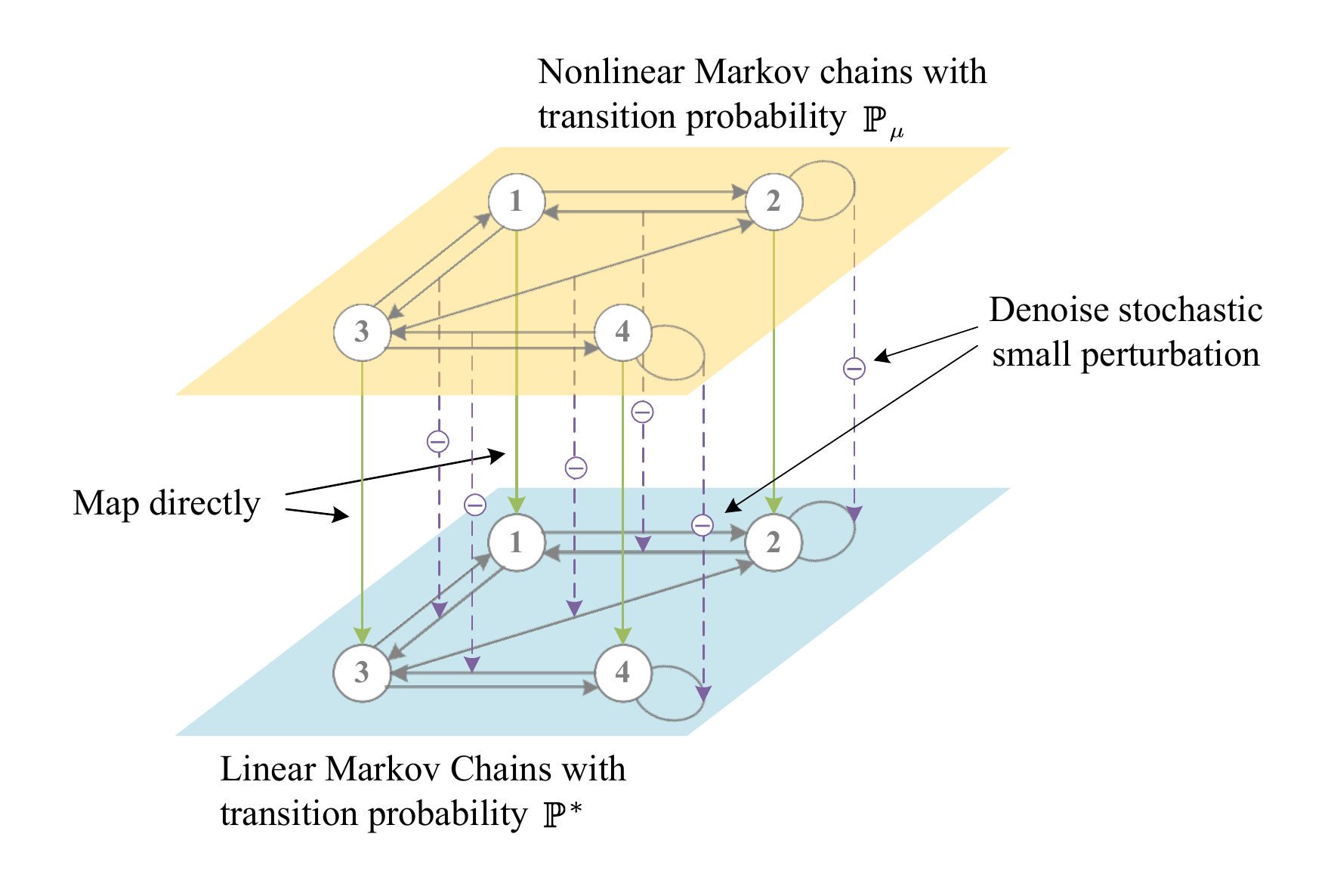}
    \caption{Map nMC to MC by Assumption 1}
    \end{minipage}
    \hspace{-1cm}
    \begin{minipage}[t]{0.48\textwidth}
    \centering
    \includegraphics[width=\linewidth]{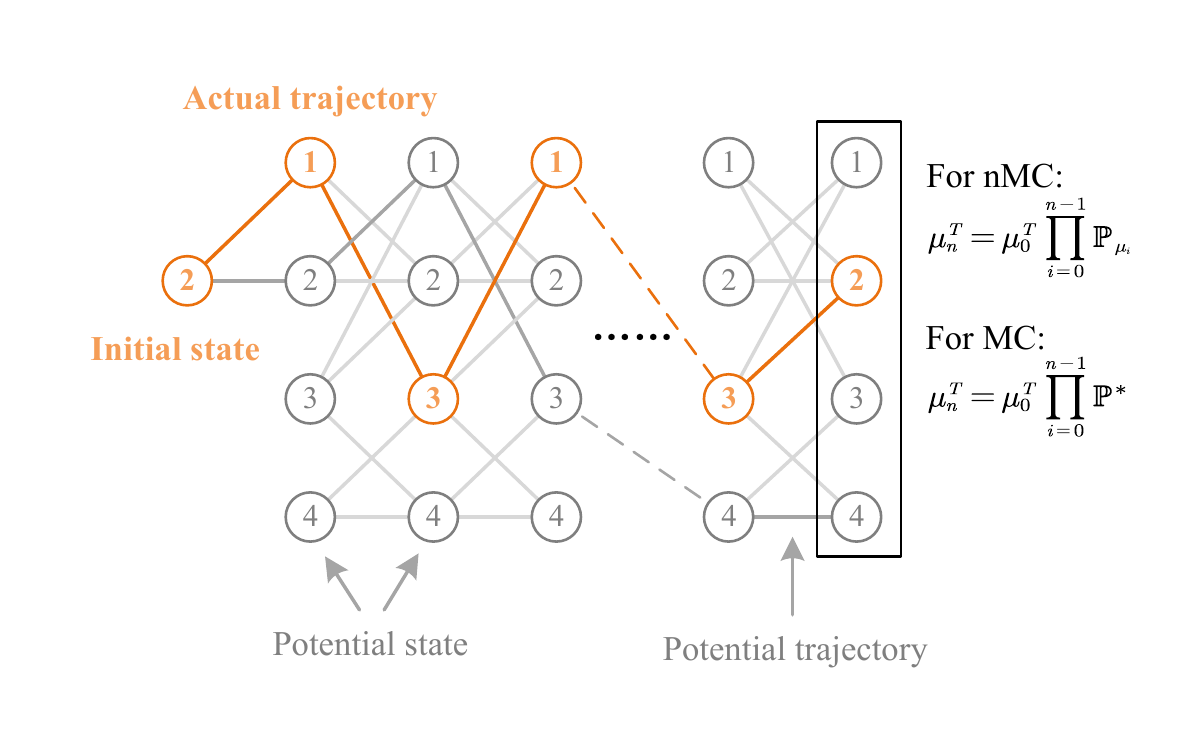}
    \caption{nMC and MC have the same trajectory}
    \end{minipage}
\end{figure}

Figure 1 and figure 2 describe these 2 assumptions in detail. Especially, the example is a nMC with 4 states. Figure 1 explains the 
idea about mapping nMC to MC by removing small perturbations. Figure 2 shows nMC and MC influence the distribution but have the same 
observed trajectory. 

The lemmas is mainly about Coupling Markov in \cite{Veretennikov1,Veretennikov2}. First, it's necessary to introduce the basic 
construction for next lemmas. Considering  two random variables $X^1$ and $X^2$ with their different measurable spaces 
$\mathscr{S}^1(E^1,\mathscr{E}^1)$ and $\mathscr{S}^2(E^2,\mathscr{E}^2)$. If the densities $p^1$ and $p^2$ on $X^1$ and $X^2$ with the 
same reference measure $\varLambda$ satisfy the condition (\ref{q}), 
\begin{equation}\label{q}
    q:= \int (p^1(x)\wedge p^2(x))\varLambda(dx) > 0
\end{equation} 
it will exist a group of random variables $\eta^1,\eta^2,\xi,\zeta$. Among them, $\zeta$ is a Bernoulli random variable, 
$\zeta \sim B(1,1-q)$, and $\eta^1,\eta^2,\xi$ have the densities (\ref{coupledef}).
\begin{equation}\label{coupledef}
    \begin{gathered}
        p^{\eta^1}(x) :=\frac{p^1-p^1\wedge p^2}{\int (p^1-p^1\wedge p^2)(y)\varLambda(dy)}(x),\quad
        p^{\eta^2}(x) :=\frac{p^2-p^1\wedge p^2}{\int (p^2-p^1\wedge p^2)(y)\varLambda(dy)}(x) \\
        p^{\xi}(x) :=\frac{p^1\wedge p^2}{\int (p^1\wedge p^2)(y)\varLambda(dy)}(x)
    \end{gathered}
\end{equation}
Especially, when $q=1$ in (\ref{q}), it sets,
\begin{equation*}
    \eta^1 = \eta^2 = \xi = X^1 = X^2
\end{equation*}
Veretennikov provides a lemma about the four random variables in \cite{Veretennikov2}. 

\begin{Lemma}\label{lemma1}
    Define a new measurable space $\mathscr{S}^* := \mathscr{S}^1 \times \mathscr{S}^2$ and two random variables 
    $(\tilde{X}^1, \tilde{X}^2)$ on $\mathscr{S}^*$, which satisfy
    \begin{equation}\label{X1X2}
        \tilde{X}^1 := \eta^1 \mathbb{I}(\zeta \neq 0) + \xi \mathbb{I}(\zeta = 0), \quad
        \tilde{X}^2 := \eta^2 \mathbb{I}(\zeta \neq 0) + \xi \mathbb{I}(\zeta = 0)
    \end{equation}
    Then, it always exists that
    \begin{equation}
        Law(\tilde{X}^j) = Law(X^j),\quad j=1,2 \quad \& \quad P(\tilde{X}^1=\tilde{X}^2)=q
    \end{equation}
\end{Lemma}

The succinct proof of lemma \ref{lemma1} can be seen in appendix \ref{appendixA}. It's worth noting that $\mathscr{S}^*$ creates a 
public measurable space for connecting two random variables. In \cite{Veretennikov2}, $(\tilde{X}^1, \tilde{X}^2)$ is called as coupling 
random variable.

For two homogeneous Markov chains $X_n^1,X_n^2$, lemma \ref{lemma1} is still correct. Define (\ref{qn}) for $X_n^1,X_n^2$ with 
distribution $\mu_n^1,\mu_n^2$
\begin{equation}\label{qn}
    q_n := \int \left(\frac{\mu_n^1(dx)}{\mu_n^2(dx)} \wedge 1\right)\mu_n^2(dx)
\end{equation}
According to the definition (\ref{coupledef}), it also exists a group of random variables $(\eta_n^1,\eta_n^2,\xi_n,\zeta_n)$. Among them,
the random sequences $\zeta_n$ are Bernoulli sequences, and the sequences $\eta_n^1,\eta_n^2,\xi_n$ satisfy:
\begin{equation}\label{mcoupledef}
    \begin{gathered}
        p^{\eta_n^1}(x):=\frac{\mu_n^1-\mu_n^1\wedge \mu_n^2}{\int (\mu_n^1-\mu_n^1\wedge \mu_n^2)(y)\varLambda(dy)}(x),\quad
        p^{\eta_n^2}(x):=\frac{\mu_n^2-\mu_n^1\wedge \mu_n^2}{\int (\mu_n^2-\mu_n^1\wedge \mu_n^2)(y)\varLambda(dy)}(x) \\
        p^{\xi_n}(x):=\frac{\mu_n^1\wedge \mu_n^2}{\int (\mu_n^1\wedge \mu_n^2)(y)\varLambda(dy)}(x)
    \end{gathered}
\end{equation}
Especially, for guaranteeing (\ref{mcoupledef}) is meaningful, when $q_n = 0$, it should set:
\begin{equation*}
    \eta_n^1 := X_n^1,\quad \eta_n^2 := X_n^2,\quad \xi_n := X_n^1,\quad \zeta_n := 1
\end{equation*}
When $q_n = 1$, it should set
\footnote{In the situation, $X_n^1 = X_n^2$. Thus, $X_n^1$ can be replaced by $X_n^2$.} :
\begin{equation*}
    \eta_n^1 := X_n^1,\quad \eta_n^2 := X_n^1,\quad \xi_n := X_n^1,\quad \zeta_n := 0
\end{equation*}
The definition (\ref{X1X2}) is used, too. Moreover, the Bernoulli random sequence $\zeta_n$ is a sign that $(\tilde{X}_n^1,\tilde{X}_n^2)$
reaches steady state $\tilde{X}_n^1 = \tilde{X}_n^2$.
\begin{equation}
    \tilde{X}_n^1 := \eta_n^1 \mathbb{I}(\zeta_n \neq 0) + \xi_n \mathbb{I}(\zeta_n = 0), \quad
    \tilde{X}_n^2 := \eta_n^2 \mathbb{I}(\zeta_n \neq 0) + \xi_n \mathbb{I}(\zeta_n = 0)
\end{equation} 

For exploring the mechanism about the changes of nMC from $n$ step to $n+1$ step, it should define the transition probability, first. 
Considering two groups of random variables $\mathcal{X}(x^1,x^2,x^3,x^4),\mathcal{Y}(y^1,y^2,y^3,y^4)$, it can define a coefficient 
$\kappa(x^1,x^2)=\int\left(\frac{P(x^1,dy)}{P(x^2,dy)}\wedge 1\right)P(x^1,dy)$. Because $\eta_n^1,\eta_n^2,\xi_n,\zeta_n$ are independent,
transition probability $P(\mathcal{X},\mathcal{Y})$ can be expressed by marginal probability.
\begin{equation}
    P(\mathcal{X},\mathcal{Y}) := P^{\eta^1}(x^1,y^1)P^{\eta^2}(x^2,y^2)P^{\xi}(x^3,y^3)P^{\zeta}(x^4,y^4)
\end{equation}
According to the definition (\ref{mcoupledef}), if $0<\kappa(x^1,x^2)<1$, the marginal probability here can be written as:
\begin{align*}
    P^{\eta^1}(x^1,y^1) &:= \frac{P(x^1,y^1)-P(x^1,y^1)\wedge P(x^2,y^1)}{1-\kappa(x^1,x^2)} \\
    P^{\eta^2}(x^2,y^2) &:= \frac{P(x^2,y^2)-P(x^1,y^2)\wedge P(x^2,y^2)}{1-\kappa(x^1,x^2)} \\
    P^{\xi}(x^3,y^3) &:= \mathbb{I}(x^4=1)\frac{P(x^1,y^3)\wedge P(x^2,y^3)}{\kappa(x^1,x^2)}+\mathbb{I}(x^4=0)P(x^3,y^3) \\
    P^{\zeta}(x^4,y^4) &:= \mathbb{I}(x^4=1)\left[\mathbb{I}(y^4=1)(1-\kappa(x^1,x^2))+\mathbb{I}(y^4=0)\kappa(x^1,x^2)\right]
    +\mathbb{I}(x^4=0)\mathbb{I}(y^4=0)
\end{align*}
If $\kappa(x^1,x^2)=0$, $P^{\xi}(x^3,y^3)$ should be reset as:
\begin{equation*}
    P^{\xi}(x^3,y^3) := \mathbb{I}(x^4=1)P(x^3,y^3) + \mathbb{I}(x^4=0)P(x^3,y^3) = P(x^3,y^3)
\end{equation*}
and if $\kappa(x^1,x^2)=1$, it should be defined as:
\begin{equation*}
    P^{\eta^1}(x^1,y^1) = P^{\eta^2}(x^2,y^2) := P(x^1,y^1)
\end{equation*} 

After that, the lemma \ref{lemma2} is proposed for describing the properties and mechanism of Coupling Markov 
$(\tilde{X}_n^1,\tilde{X}_n^2)$.

\begin{Lemma}\label{lemma2}
    If it exists two homogeneous Markov chains $(X_n^1),(X_n^2)$ with the same state space $\mathscr{S}$ and transition probability matrix
    $\mathbb{P}$, $(\tilde{X}_n^1,\tilde{X}_n^2)$ is a couple of Markov chains which is the same as $(X_n^1),(X_n^2)$ in transition 
    probability. Moreover, for any $n \in \mathbb{Z}_+$, it always has:
    \begin{equation}
        Law(\tilde{X}_n^j) = Law(X_n^j),\quad j=1,2 \quad \& \quad P(\tilde{X}_n^1=\tilde{X}_n^2)=q_n
    \end{equation}
\end{Lemma}

The proof of lemma \ref{lemma2} will be seen in appendix \ref{appendixB}. It provides an important idea to re-express the irrelevant Markov
chains by a couple of relevant Markov chains. They have the same generation approach, transition probability and even distribution at any
time steps. 

Additionally, for $\mathscr{X}_n=(\tilde{X}_n^1,\tilde{X}_n^2)$, when it doesn't reach a coupling state $\tilde{X}_n^1\neq \tilde{X}_n^2$,
$\mathscr{X}_n$ is on $\mathscr{S}^{p(p-1)}$. $\mathscr{X}_n$ transforms to $\mathscr{X}_{n+1}$, which means that
$(\eta_n^1,\eta_n^2,*,1)$ transforms to $(\eta_{n+1}^1,\eta_{n+1}^2,*,1)$ where $*$ is expressed for any values. The transition probability
and matrix can be defined as:
\begin{equation}\label{M}
    \begin{gathered}
        P(\mathscr{X}_n,\mathscr{X}_{n+1})=P^{\eta^1}(\eta_n^1,\eta_{n+1}^1)P^{\eta^2}(\eta_n^2,\eta_{n+1}^2)P^{\zeta}(1,1) \\
        \\
        \mathbb{M}= 
        \begin{pmatrix}
            P(\mathscr{X}_1,\mathscr{X}_1) & P(\mathscr{X}_1,\mathscr{X}_2) & \cdots & P(\mathscr{X}_1,\mathscr{X}_{p(p-1)}) \\
            P(\mathscr{X}_2,\mathscr{X}_1) & P(\mathscr{X}_2,\mathscr{X}_2) & \cdots & P(\mathscr{X}_2,\mathscr{X}_{p(p-1)}) \\
            \vdots & \vdots & \ddots & \vdots \\
            P(\mathscr{X}_{p(p-1)},\mathscr{X}_1) & P(\mathscr{X}_{p(p-1)},\mathscr{X}_2) & \cdots & P(\mathscr{X}_{p(p-1)},\mathscr{X}_{p(p-1)})
        \end{pmatrix}
    \end{gathered}
\end{equation}

\section{Main Result and Examples}\label{sec3}
Considering nMC $(X_n)_{n \in \mathbb{Z}_+}$ and the mapping linear MC $(X_n^*)_{n \in \mathbb{Z}_+}$, the TV distance between $\mu_n$ and
$\pi$ can be limited by using the triangle inequality. 
\begin{equation}\label{triangle}
    \|\mu_n-\pi\|_{TV} \leq \|\mu_n-\mu_n^*\|_{TV} + \|\pi-\pi^*\|_{TV} + \|\mu_n^*-\pi^*\|_{TV}
\end{equation}
Then, every parts in the inequality will be compressed in order to obtain the enhanced convergence bounds. 

\begin{Theorem}\label{T1}
    For two linear MC $({X_n^1}^*)_{n \in \mathbb{Z}_+}$ with distribution $\mu_n^*$ and $({X_n^2}^*)_{n \in \mathbb{Z}_+}$ with steady
    distribution $\pi^*$, there is a coupling Markov $(\tilde{X}_n^1,\tilde{X}_n^2)$ with transition matrix $\mathbb{M}$. And it always 
    exists: 
    \begin{equation}
        \|\mu_n^*-\pi^*\|_{TV} \leq \left(r(\mathbb{M})+\varepsilon\right)^n \|\mu_0^*-\pi^*\|_{TV}
    \end{equation}
    Here $r(\mathbb{M})$ is the spectral radius, and $\varepsilon$ is a small quantity. 
\end{Theorem}

\begin{proof}
    For any $n \in \mathbb{Z}_+$, by using (\ref{qn}), it has
    \begin{equation}\label{leq1}
    \begin{aligned}
        \frac{1}{2}\|\mu_n^*-\pi^*\|_{TV} &= \frac{1}{2}\sum_x |\mu_n^*(x)-\pi^*(x)| \\
        &= \frac{1}{2}\left(\sum_x\mu_n^*(x)+\sum_x\pi^*(x)-2\sum_x\mu_n^*\wedge\pi(x)\right) \\
        &= 1 - \sum_x \left(\frac{\pi(x)}{\mu_n^*(x)}\wedge 1\right)\mu_n^*(x)=1-q_n \\
        &=1-P_{\mu_0^*,\pi^*}(\zeta_n=0) = P_{\mu_0^*,\pi^*}(\tilde{X}_n^1 \neq \tilde{X}_n^2) \\
    \end{aligned}
    \end{equation}
    Recall the transition probability of $\zeta_n$. If $\exists n_0$, $\zeta_{n_0}=0$, then $\forall n > n_0$, $\zeta_n=0$. So, for 
    guaranteeing $\tilde{X}_n^1 \neq \tilde{X}_n^2$, $\forall n_0 \in \mathbb{Z}_+,n_0 \leq n$, it must have $\zeta_{n_0}=1$.
    Moreover, $\zeta_n$ depends on $\eta_n^1,\eta_n^2$. There are many kinds of trajectories satisfying the condition. Considering all 
    possibilities, it uses expection $\mathbb{E}$ about all the trajectories and $\mathscr{X}_i=(\eta_i^1,\eta_i^2,*,1)$ for simple 
    expression, and probability can be written as  
    \begin{equation}\label{leq2}
    \begin{aligned}
        P_{\mu_0^*,\pi^*}(\tilde{X}_n^1 \neq \tilde{X}_n^2) &= \mathbb{E}_{\mu_0^*,\pi^*}
        \left(\prod_{i=0}^{n-1}P(\mathscr{X}_i,\mathscr{X}_{i+1})\mathbb{I}(\eta_{i+1}^1\neq\eta_{i+1}^2)
        \mathbb{I}(\tilde{X}_0^1\neq\tilde{X}_0^2)\right) \\
        &=\mathbb{E}_{\mu_0^*,\pi^*}
        \left(\prod_{i=0}^{n-1}P(\mathscr{X}_i,\mathscr{X}_{i+1})\mathbb{I}(\eta_{i+1}^1\neq\eta_{i+1}^2)\right)
        P_{\mu_0^*,\pi^*}(\tilde{X}_0^1\neq\tilde{X}_0^2) \\
        &=\sum_{\mathscr{X}_0}\sum_{\mathscr{X}_{n-1}}\left(\mathbb{M}^n\right)_{\mathscr{X}_0,\mathscr{X}_{n-1}}
        p(\mathscr{X}_0)*\left(\frac{1}{2}\|\mu_0^*-\pi\|_{TV}\right) \\
        &\leq \sum_{\mathscr{X}_0}\|\mathbb{M}^n\|_1p(\mathscr{X}_0)*\left(\frac{1}{2}\|\mu_0^*-\pi\|_{TV}\right) \\
        &\leq \|\mathbb{M}\|_1^n \left(\frac{1}{2}\|\mu_0^*-\pi\|_{TV}\right) 
    \end{aligned}    
    \end{equation}
    where $\|\mathbb{M}\|_1$ is a 1-norm for matrix. Because the elements from $\mathbb{M}$ are positive, 
    \begin{equation*}
        \|\mathbb{M}\|_1 = \max_{\mathscr{X}_i} \sum_{\mathscr{X}_j} |\mathbb{M}_{\mathscr{X}_i,\mathscr{X}_j}| =
        \max_{\mathscr{X}_i} \sum_{\mathscr{X}_j} \mathbb{M}_{\mathscr{X}_i,\mathscr{X}_j}
    \end{equation*}
    Thus, according to (\ref{leq1}) and (\ref{leq2}), it gets
    \begin{equation*}
        \|\mu_n^*-\pi^*\|_{TV} \leq \|\mathbb{M}\|_1^n \|\mu_0^*-\pi^*\|_{TV} \leq \left(r(\mathbb{M})+\varepsilon\right)^n
        \|\mu_0^*-\pi^*\|_{TV}
    \end{equation*}
\end{proof}

\begin{Theorem}\label{T2}
    The mapping linear Markov chain $(X_n^*)_{n \in \mathbb{Z}_+}$ has a state space $\mathscr{S}(E,\mathscr{E})$ and the kinds of state
    $p\geq 2$. If the initial distribution $\mu_0^*$ can be set in advance and the limiting distribution $\pi^*$ is unknown, the inequality 
    (\ref{mu0}) will always hold.
    \begin{equation}\label{mu0}
        \|\mu_0^*-\pi^*\|_{TV} \leq 2(1-\frac{1}{p})
    \end{equation}
\end{Theorem}

\begin{proof}
    Considering a discrete uniform p-dimension distribution $\tilde{\mu} \in \mathcal{P}(E)$, it can be expressed as 
    $\tilde{\mu}=(\tilde{\mu}_1,\tilde{\mu}_2,\cdots,\tilde{\mu}_p)^T=(\frac{1}{p},\frac{1}{p},\cdots,\frac{1}{p})^T$. According to the 
    triangle inequality, apparently, 
    \begin{equation}\label{triangle1}
        \|\mu_0^*-\pi^*\|_{TV} \leq \|\mu_0^*-\tilde{\mu}\|_{TV} + \|\tilde{\mu}-\pi^*\|_{TV}
    \end{equation}
    If it sets $\mu_0^* = \tilde{\mu}$, $\|\mu_0^*-\tilde{\mu}\|_{TV}$ will reach the minimum 0. Because 
    $\pi^*=(\pi_1^*,\pi_2^*,\cdots,\pi_p^*)^T$, the maximum of $\|\tilde{\mu}-\pi^*\|_{TV}$ can be described as the optimization 
    (\ref*{op1}).
    \begin{equation}\label{op1}
    \begin{aligned}
        \max_{\pi^*} \quad & \|\tilde{\mu}^*-\pi^*\|_{TV} = \sum_{i=1}^{p}\left |\frac{1}{p}-\pi_i^*\right | \\
        \mbox{s.t.}\quad & \sum_{i=1}^{p}\pi_i^* = 1,\quad \forall \pi^* \in \mathcal{P}(E) \\
        & 0 \leq \pi_i^* \leq 1,\quad i=1,2,\cdots,p
    \end{aligned}   
    \end{equation}
    To eliminate the effect of absolute values, there are two variables $\delta_i^+,\delta_i^-$ proposed. 
    \begin{equation*}
        \delta_i^+ := \max \left \{0,\pi_i^*-\frac{1}{p} \right \},\quad
        \delta_i^- := \max \left \{0,\frac{1}{p}-\pi_i^* \right \}
    \end{equation*}
    When $\pi_i^*-1/p \geq 0$, $\delta_i^+=\pi_i^*-1/p,\delta_i^-=0$. When $\pi_i^*-1/p < 0$, $\delta_i^+=0,\delta_i^-=1/p-\pi_i^*$. 
    So, there are two important properties about $\delta_i^+,\delta_i^-$. 
    \begin{equation*}
        \left |\pi_i^*-\frac{1}{p} \right | = \delta_i^++\delta_i^-, \quad
        \pi_i^*-\frac{1}{p} = \delta_i^+-\delta_i^-
    \end{equation*}
    (\ref{op1}) can be re-expressed by these two variables. It always exists
    \begin{equation}\label{d1d2_1}
        \sum_{i=1}^{p}\left |\frac{1}{p}-\pi_i^*\right | = \sum_{i=1}^{p}\left(\delta_i^++\delta_i^-\right)
    \end{equation}
    \begin{equation}\label{d1d2_2}
        \sum_{i=1}^{p}\pi_i^*=1+\sum_{i=1}^{p}\left(\pi_i^*-\frac{1}{p}\right)=1+\sum_{i=1}^{p}\left(\delta_i^+-\delta_i^-\right)=1
    \end{equation}
    \begin{equation}\label{d1d2_3}
        0 \leq \delta_i^+-\delta_i^-+\frac{1}{p} \leq 1,\quad \delta_i^+,\delta_i^-\geq 0
    \end{equation}
    \begin{equation}\label{d1d2_4}
        \delta_i^+\delta_i^- = \left((\pi_i^*-\frac{1}{p})*0\right)\mathbb{I}(\pi_i^*-1/p \geq 0) + 
        \left(0*(\frac{1}{p}-\pi_i^*)\right)\mathbb{I}(1/p-\pi_i^* < 0)=0
    \end{equation}
    By the conditions (\ref{d1d2_1})-(\ref{d1d2_4}), the original optimization (\ref{op1}) will have other equivalent form.
    \begin{equation}\label{op2}
    \begin{aligned}
        \max_{\delta^+,\delta^-} \quad & f(\delta^+,\delta^-) = \sum_{i=1}^{p}\left(\delta_i^++\delta_i^-\right) \\
        \mbox{s.t.}\quad & \sum_{i=1}^{p}\left(\delta_i^+-\delta_i^-\right)=0 \\
        & 0 \leq \delta_i^+ \leq 1-\frac{1}{p},\quad 0 \leq \delta_i^- \leq \frac{1}{p} \\ 
        & \delta_i^+\delta_i^-=0,\quad i=1,2,\cdots,p
    \end{aligned}
    \end{equation}
    Let $m,n$ count the number of $\delta_i^+>0$ and $\delta_i^->0$. Because (\ref{d1d2_4}), $m+n \leq p$. 
    \begin{equation*}
        m:=\sum_{i=1}^p\mathbb{I}(\delta_i^+>0),\quad n:=\sum_{i=1}^p\mathbb{I}(\delta_i^->0)
    \end{equation*}
    When $m+n \leq p-2$, $\exists j,k$, $\delta_i^+=\delta_i^-=0$ where $i=j,k$. So, a better solution $(\delta_i^+,\delta_i^-)$
    can be found by the initial solution $(\delta_i^{+*},\delta_i^{-*})$. It's worth noting that $1/p \leq 1 - 1/p$, 
    $(\delta_i^+,\delta_i^-)$ can satisfy (\ref{op2}) and have a larger $f(\delta^+,\delta^-)$. Thus, the best solution won't be in this 
    case. 
    \begin{equation*}
        (\delta_i^+,\delta_i^-)=
        \begin{cases}
            (\delta_i^{+*},\delta_i^{-*}),\quad &i\neq j,k\\
            (\frac{1}{p},0),&i=j \\
            (0,\frac{1}{p}),&i=k
        \end{cases}
    \end{equation*}
    When $m+n=p-1$, $\sum_{i=1}^{p}\delta_i^+$ and $\sum_{i=1}^{p}\delta_i^-$ have the upper bround $m(1-1/p)$ and $n/p$.
    \begin{equation*}
        \frac{n}{p}=\frac{p-1}{p}-\frac{m}{p} \leq m-\frac{m}{p} = m\left(1-\frac{1}{p}\right)
    \end{equation*}
    However, $\sum_{i=1}^{p}\delta_i^+ = \sum_{i=1}^{p}\delta_i^- \leq n/p$. And $\exists j,0<\delta_j^+<1-1/p$, 
    $\exists k,\delta_k^+=\delta_k^-=0$. It sets a small quantity $\varepsilon$. A better solution can be found with the similar idea. 
    So, the best solution won't also be in this case. 
    \begin{equation*}
        (\delta_i^+,\delta_i^-)=
        \begin{cases}
            (\delta_i^{+*},\delta_i^{-*}),\quad &i\neq j,k\\
            (\delta_i^{+*}+\varepsilon,0),&i=j \\
            (0,\varepsilon),&i=k
        \end{cases}
    \end{equation*}
    Importantly, when $m+n=p$, it always exists $m(1-1/p)\geq n/p$. So, in the best solution, 
    $\sum_{i=1}^{p}\delta_i^+ = \sum_{i=1}^{p}\delta_i^- = n/p$. Because $m\geq 1,n\leq p-1$, the maximum of $f(\delta^+,\delta^-)$ holds:
    \begin{equation*}
        \max_{\delta^+,\delta^-}f(\delta^+,\delta^-) =\max_{n} \frac{2n}{p} = 2\left(1-\frac{1}{p}\right)
    \end{equation*}
    When $\pi^*=(1,0,\cdots,0)^T$, the maximum will be reached. Thus, according to (\ref{triangle1}), (\ref{mu0}) has been demonstrated.
\end{proof}

The bounds of $\|\mu_n^*-\pi^*\|_{TV}$ in (\ref{triangle}) have been estimated so far. Next, it will pay attention to the TV distance 
between nMC and the mapping linear MC. 

\begin{Theorem}\label{T3}
    For any $n$, the TV distance $\|\mu_n-\mu_n^*\|_{TV}$ will have a upper bound. Moreover, it can expressed by
    \begin{equation}
        \|\mu_n-\mu_n^*\|_{TV} \leq 2\frac{e^{-(1+n\gamma)}+n\gamma}{1+n\gamma}
    \end{equation}
    where $\gamma$ is a small coefficient defined by (\ref{assumption1}).
\end{Theorem}

\begin{proof}
    First, $\rho_n$ will be defined by (\ref{rho}), which is the most important in the proof.
    \begin{equation}\label{rho}
        \rho_n := \prod_{i=0}^{n-1}\frac{P^*(\tilde{x}_i^1,\tilde{x}_{i+1}^1)}{P_{\mu_i}(\tilde{x}_i^1,\tilde{x}_{i+1}^1)}
    \end{equation}
    Actually, $(\tilde{x}_n^1)$ is the observed trajectory of nMC. When $\tilde{x}_i^1$ is fixed, $\tilde{x}_{i+1}^1$ can be regarded as
    a sample from $\mu_{i+1}$. By the way, $\tilde{x}_i^1$ is one of the coupling Markov $(\tilde{X}_i^1,\tilde{X}_i^2)$. But in this case, 
    these Markov chains aren't homogeneous, which means the theory about $\mathbb{M}$ can't be uesd. Considering a trajectory $\mathcal{A}$ 
    from 0 to n step on the state space $\mathscr{S}^n$, the $k$ th $(k\geq 1)$ moment of $\rho_n$ holds
    \begin{align*}
        \mathbb{E}(\rho_n^k) &= \sum_{\mathcal{A}}\left(
            \prod_{i=0}^{n-1}\frac{P^*(\tilde{x}_i^1,\tilde{x}_{i+1}^1)}{P_{\mu_i}(\tilde{x}_i^1,\tilde{x}_{i+1}^1)}
        \right)^k\left(
            \prod_{i=0}^{n-1} P_{\mu_i}(\tilde{x}_i^1,\tilde{x}_{i+1}^1)
        \right)\mu_0(\tilde{x}_0^1)\\
        &=\sum_{\mathcal{A}}\prod_{i=0}^{n-1}\left[
            \left(
                \frac{P^*(\tilde{x}_i^1,\tilde{x}_{i+1}^1)}{P_{\mu_i}(\tilde{x}_i^1,\tilde{x}_{i+1}^1)}
            \right)^{k-1}P^*(\tilde{x}_i^1,\tilde{x}_{i+1}^1)
        \right]\mu_0(\tilde{x}_0^1)\\
        &\leq \sum_{\mathcal{A}}\prod_{i=0}^{n-1} (1+\gamma)^{k-1}P^*(\tilde{x}_i^1,\tilde{x}_{i+1}^1)\mu_0(\tilde{x}_0^1) \\
        &= \sum_{\tilde{x}_0^1}(1+\gamma)^{n(k-1)}\mu_0(\tilde{x}_0^1) = (1+\gamma)^{n(k-1)}
    \end{align*}
    Second, the definition (\ref{rho}) will be used for re-expression of the TV distance. 
    \begin{align*}
        \frac{1}{2}\|\mu_n-\mu_n^*\|_{TV} &= 1-\sum_x\left(1\wedge \frac{\mu_n^*(x)}{\mu_n(x)}\right)\mu_n(x) \\
        &= 1-\sum_A \left(
            1\wedge \prod_{i=0}^{n-1}\frac{P^*(\tilde{x}_i^1,\tilde{x}_{i+1}^1)}
            {P_{\mu_i}(\tilde{x}_i^1,\tilde{x}_{i+1}^1)}
            \right)\prod_{i=0}^{n-1}P_{\mu_i}(\tilde{x}_i^1,\tilde{x}_{i+1}^1)\mu_0(\tilde{x}_0^1) \\
        & = 1-\mathbb{E}(1\wedge \rho_n) = \mathbb{E}[(1-\rho_n)\mathbb{I}(\rho_n < 1)]
    \end{align*}
    $\rho_n \in (0,(1+\gamma)^n)$, but $\mathbb{E}[(1-\rho_n)\mathbb{I}(\rho_n < 1)]$ will limit $\rho_n$. Please pay attention to
    \begin{equation*}
        1-\rho_n \leq e^{-\rho_n},\quad \forall \rho_n \in (0,+\infty)
    \end{equation*}
    So, the inequality will be used to relax the bounds. And an important result can be provided. 
    \begin{align*}
        \mathbb{E}[(1-\rho_n)\mathbb{I}(\rho_n < 1)] &\leq \mathbb{E}(e^{-\rho_n}) = 
        \mathbb{E}\left(\sum_{i=0}^\infty \frac{(-\rho_n)^i}{i!}\right) = \sum_{i=0}^\infty 
        \left(\mathbb{E}\frac{(-\rho_n)^i}{i!}\right) \\
        &= \sum_{i=2}^\infty\frac{(-1)^i}{i!}\mathbb{E}(\rho_n^i) \leq \sum_{i=2}^\infty\frac{(-1)^i}{i!}(1+\gamma)^{n(i-1)}\\
        &= \frac{1}{(1+\gamma)^n}\left(e^{-(1+\gamma)^n}+(1+\gamma)^n-1\right)
    \end{align*}
    Because $\gamma$ is a small quantity, $(1+\gamma)^n \approx 1+n\gamma$ for any $n$. Then, the theorem has been demonstrated. 
\end{proof}

Because nMC is regarded as MC with small nonlinear perturbations, let $\delta$ be a small quantity, the TV distance between two limiting
distributions can be set as:
\begin{equation*}
    \|\pi-\pi^*\|_{TV} \leq \delta
\end{equation*}
Then, according to the Theorem \ref{T1},\ref{T2},\ref{T3}, the main resluts about convergence bounds has be provided as the following
Theorem.

\begin{Theorem}
    For the nonlinear Markov chain with small perturbations, if it exists the convergence bounds, it will be estimated by this method. 
    When the time step $n$ isn't large enough, the bounds hold,
    \begin{equation}
        \|\mu_n-\pi\|_{TV} \leq 2e^{-1} + \delta + 2(r(\mathbb{M})+\varepsilon)^n\left(1-\frac{1}{p}\right)
    \end{equation}
    And when $n$ is large enough, it will exist,
    \begin{equation}
        \|\mu_n-\pi\|_{TV} \leq 2\delta + 2(r(\mathbb{M})+\varepsilon)^n\left(1-\frac{1}{p}\right)
    \end{equation}
    If it doesn't need to estimate precisely, $\delta$ can be set to 0.
\end{Theorem}

There are two numerical examples for explaining and presenting the main results, which are the same examples in \cite{Shchegolev3}. 
Moreover, in the two examples, the transition matrix will be provided, but the initial distribution is unknown. So, let $\mu_0$ be a random
distribution defined on $\mathscr{S}^p$, with nMC $(X_n)_{n \in \mathbb{Z}_+}$ defining on $\mathscr{S}^p(E,2^E)$. The different nMC are  
simulated for 1000 times. 

\begin{Example}
    Considering a discrete nMC defined on $\mathscr{S}(\{1,2,3,4\},2^{\{1,2,3,4\}})$ with distribution 
    $\mu_n=(\mu_n^1,\mu_n^2,\mu_n^3,\mu_n^4)^T$, the nMC transition probability matrix will be expressed
    \begin{equation*}
        \mathbb{P}_{\mu_i} = 
        \begin{pmatrix}
           0.4-\kappa\mu_i^1 & 0.2 & 0.2+\kappa\mu_i^1 & 0.2 \\
           0.3 & 0.4 & 0.2 & 0.1 \\
           0.2 & 0.2 & 0.4 & 0.2 \\
           0.2 & 0.1 & 0.2 & 0.5  
        \end{pmatrix},\quad
        \mathbb{P} = 
        \begin{pmatrix}
           0.4 & 0.2 & 0.2 & 0.2 \\
           0.3 & 0.4 & 0.2 & 0.1 \\
           0.2 & 0.2 & 0.4 & 0.2 \\
           0.2 & 0.1 & 0.2 & 0.5 
        \end{pmatrix}
    \end{equation*}
    The mapping Markov chains can be provided by the transition matrix $\mathbb{P}^*$. For coupling markov, the transition probability
    matrix $\mathbb{M}_{12 \times 12}$ can be calculated by (\ref{M}). 

    According to (\ref{alpha}), the M-D condition bounds can be obtained. 
    $2(1-\alpha_1)=0.8,\,2(1-\alpha_2)=0.3,\,2(1-\alpha_3)=0.11,\,2(1-\alpha_4)=0.04,\,\cdots$
    And the new bounds estimated by our spectral radius method can hold
    $2(1-1/p)(r(\mathbb{M})+\varepsilon)^1=0.54,\,2(1-1/p)(r(\mathbb{M})+\varepsilon)^2=0.20,\,2(1-1/p)(r(\mathbb{M})+\varepsilon)^3=0.07,\,\cdots$
    For further test, the nonlinear coefficient $\kappa$ will be set as 0.1 and 0.2 for simulation in figure 3. 
    \begin{figure}[htb]
        \centering
        \captionsetup[subfigure]{labelformat=empty}
        \begin{subfigure}{.48\textwidth}
            \includegraphics[width=\linewidth]{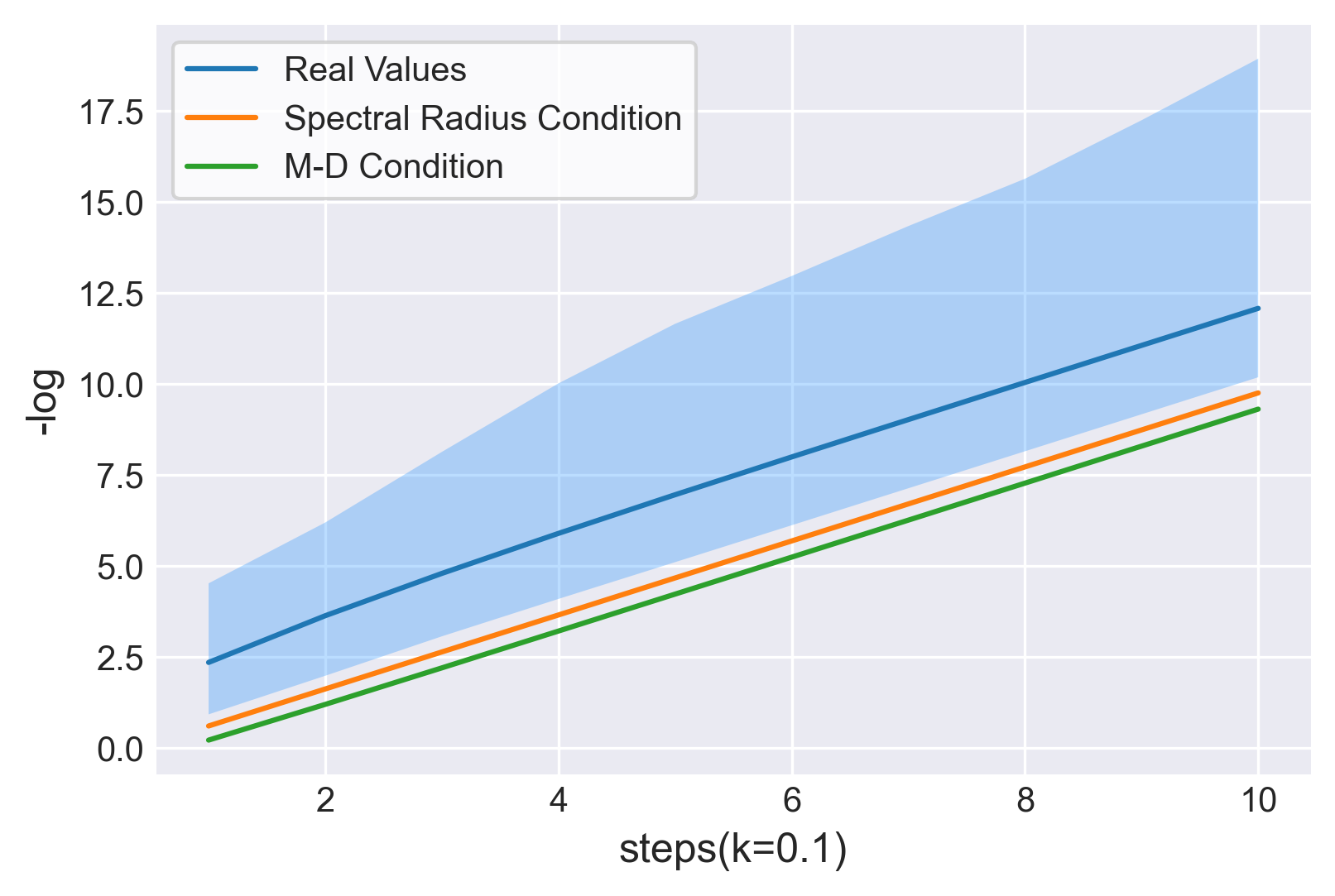}
        \end{subfigure}
        \begin{subfigure}{.48\textwidth}
            \includegraphics[width=\linewidth]{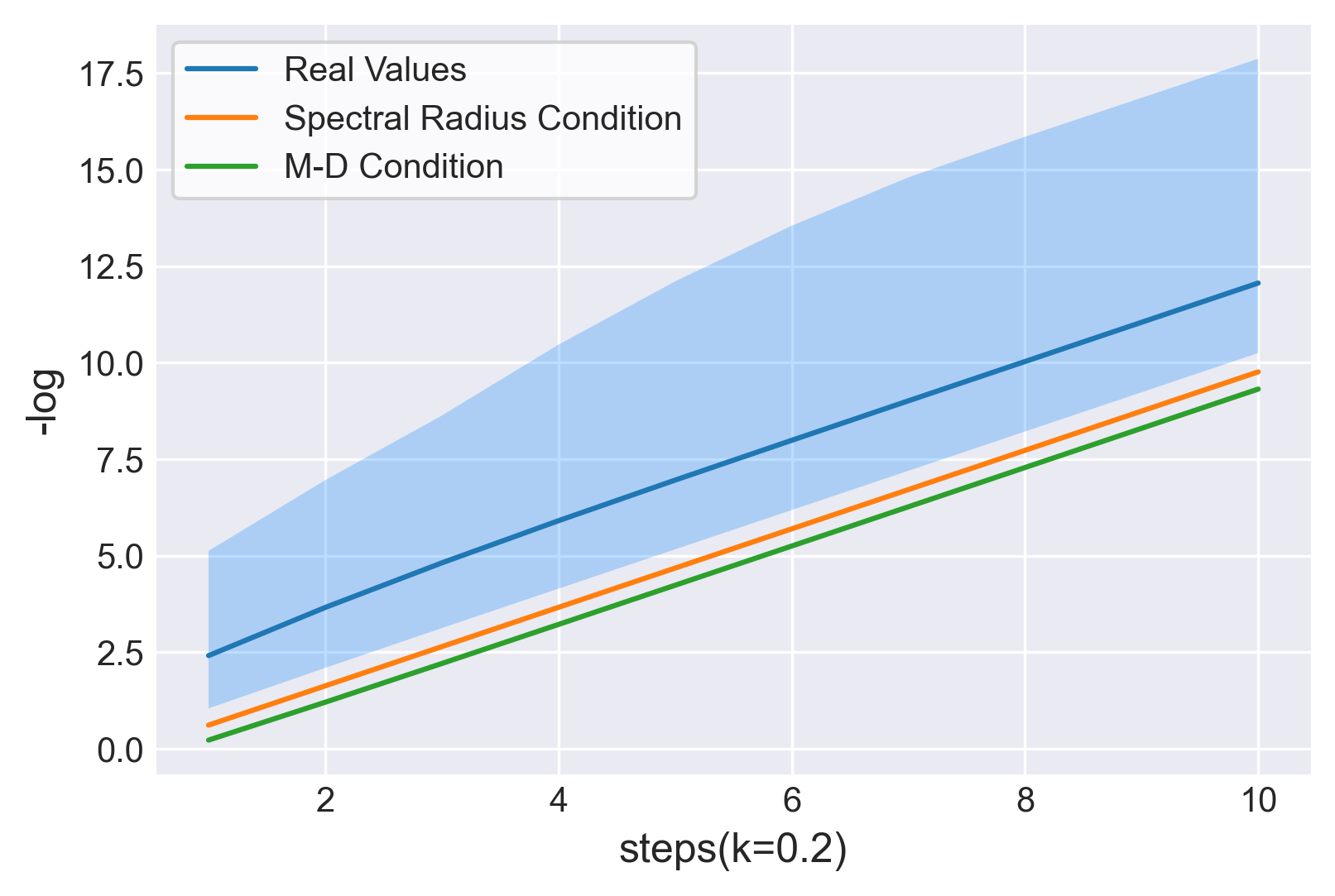}
        \end{subfigure}
        \caption{Estimated convergence bounds and interval about the true values in example 1}
    \end{figure}
\end{Example}

\begin{Example}
    On the basis of example 1, a more complex example will be researched. The state space $\mathscr{S}(\{1,2,3,4,5\},2^{\{1,2,3,4,5\}})$, 
    transition probability matrix $\mathbb{P}_{\mu_i},\mathbb{P}^*$ of nMC and linear MC hold
    \begin{gather*}
        \mathbb{P}_{\mu_i}=
        \begin{pmatrix}
            0.4+\kappa\mu_i^1 & 0.3-\kappa\mu_i^1 & 0.1 & 0.1 & 0.1 \\
            0.2 & 0.4+\kappa\mu_i^2 & 0.2-\kappa\mu_i^2 & 0.1 & 0.1 \\
            0.1 & 0.2 & 0.4+\kappa\mu_i^3 & 0.2-\kappa\mu_i^4 & 0.1 \\
            0.1 & 0.1 & 0.2 & 0.4+\kappa\mu_i^4 & 0.2-\kappa\mu_i^4 \\
            0.1 & 0.1 & 0.1 & 0.3-\kappa\mu_i^5 & 0.4+\kappa\mu_i^5
        \end{pmatrix} \\
        \\
        \mathbb{P}^*=
        \begin{pmatrix}
            0.4 & 0.3 & 0.1 & 0.1 & 0.1 \\
            0.2 & 0.4 & 0.2 & 0.1 & 0.1 \\
            0.1 & 0.2 & 0.4 & 0.2 & 0.1 \\
            0.1 & 0.1 & 0.2 & 0.4 & 0.2 \\
            0.1 & 0.1 & 0.1 & 0.3 & 0.4
        \end{pmatrix}
    \end{gather*}
    Similarly, it's easy to recognize the spectral radius bounds is smaller than the M-D bounds. The nonlinear coefficient $\kappa$ is 
    also set to 0.1 and 0.2. The comparison is shown in figure 4.
    \begin{figure}[htb]
        \centering
        \captionsetup[subfigure]{labelformat=empty}
        \begin{subfigure}{.48\textwidth}
            \includegraphics[width=\linewidth]{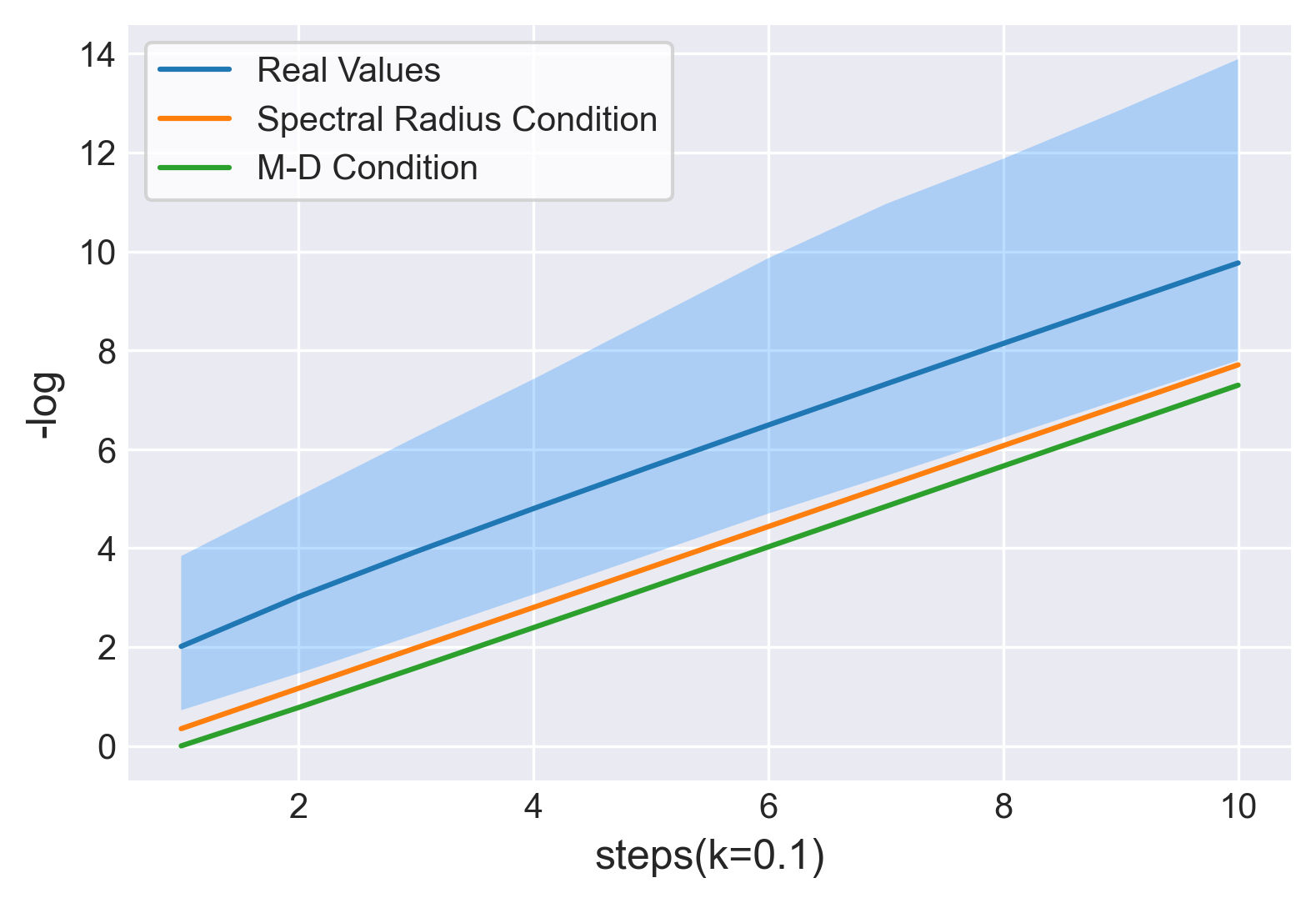}
        \end{subfigure}
        \begin{subfigure}{.48\textwidth}
            \includegraphics[width=\linewidth]{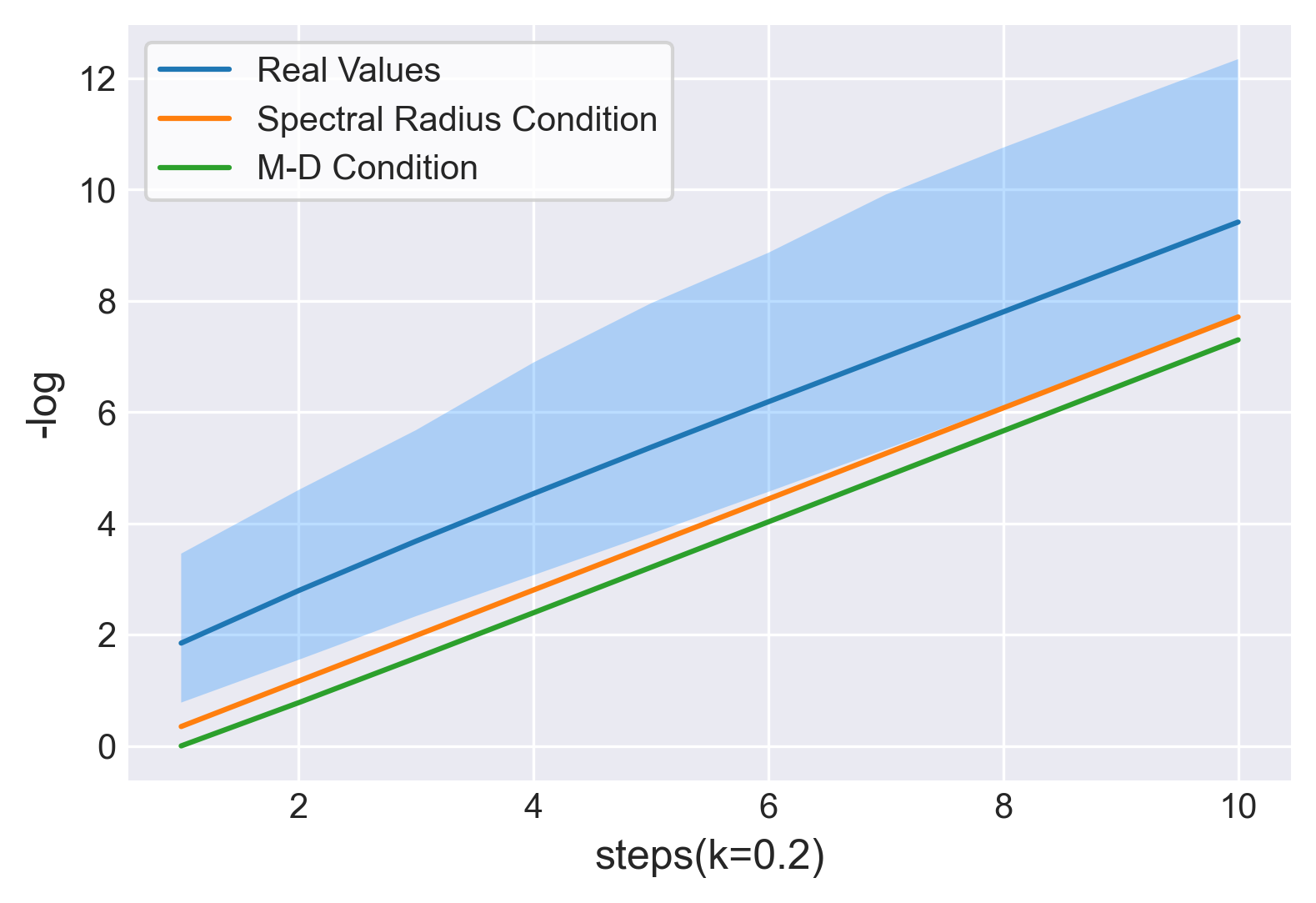}
        \end{subfigure}
        \caption{Estimated convergence bounds and interval about the true values in example 2}
    \end{figure}
\end{Example}

Through these two examples, it has been verified to the main result about convergence bounds of nMC. In addition, the spectral radius 
bounds are better than M-D bounds. In some situations, spectral radius method has even clung to the lower bound of the real TV distance. 
So, the main result can be regarded as an precise enough approach.

\section{Application}\label{sec4}
The TV distance is a kind of distance between two different distributions. If $\|\mu_n-\pi\|_{TV}=0$, the stochastic process will reach
a statistical steady state. If the convergence bounds are large, the volatility of time series will be also large, and it needs a lot of 
time to become steady. Thus, by the simple idea, the convergence bounds through special radius can measure the volatility of series. 
Then, a specific algorithm is provided as follows. 

\begin{figure}[htbp]
    \centering
    \includegraphics[width=\linewidth]{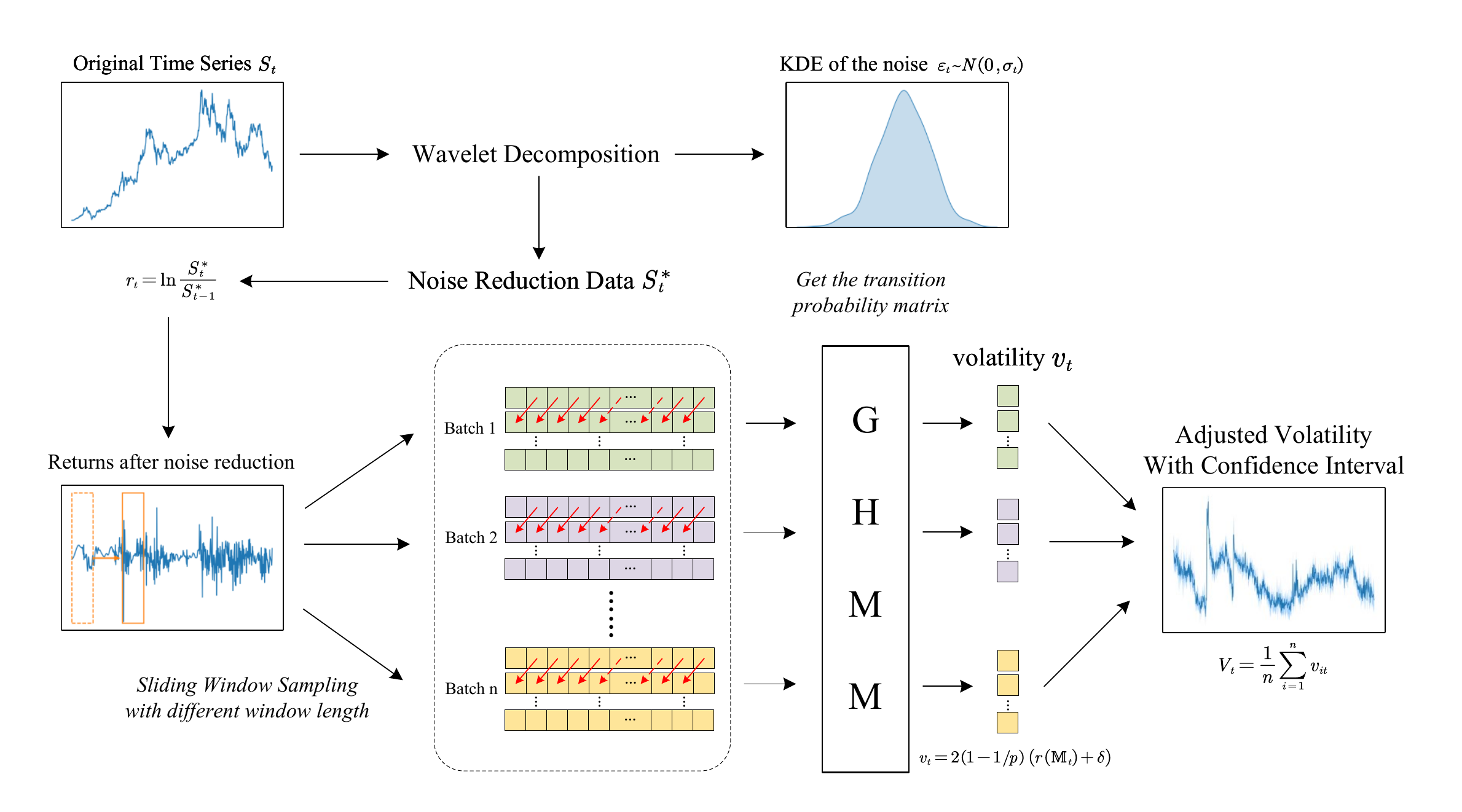}
    \caption{The framework about estimating the volatility of time series through the convergence bounds}
\end{figure} 

Figure 5 shows the framework of the approach to estimate the volatility of securities. First, when it accepts a original daily price 
series $S_t$, the series will be reconstructed through Wavelet analysis for removing the noise mixed in the data. Then, the returns can 
be calculated by the noise reduction data $S_t^*$ according to $r_t = \frac{\ln S_t^*}{\ln S_{t-1}^*}$. For getting a more credible 
result, the sliding windows sampling with different windows length is applied to get many batches of series. These batches can be sent 
to Gaussian Hidden Markov Model (GHMM) to estimate the transition probability matrix. It's worth noting that the convergence bounds 
method provided in section \ref{sec3} always need a markov chain with discontinuous states. Although $r_t$ is continuous, GHMM can 
transform it to discrete hidden sequences $X_t$. At least, the volatility can be provided by the mean of the estimation from different 
batches. If the process is repeated for many times, it will get the confidence interval of the volatility. Moreover, the changes of 
securities price have a complex mechanism. The traditional methods always need a couple of assumptions or can't fit the real situation 
very well. But nMC can relax the restrictions (\ref{def1}), it's ideal for dealing with this type of problem.

\begin{figure}[htbp]
    \centering
    \includegraphics[width=0.8\linewidth]{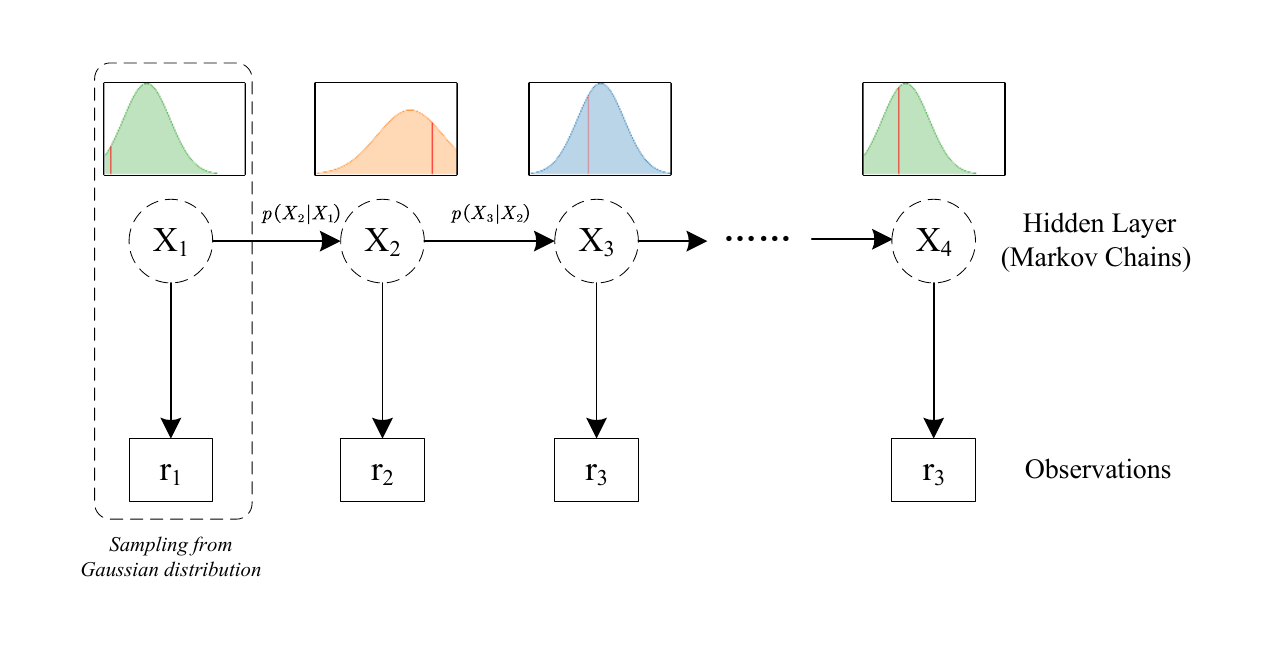}
    \caption{Principle about the Gaussian Hidden Markov Model}
\end{figure}

In GHMM, there are two different components, hidden layer and observations. Hidden layer is a markov chain with discrete states. Every 
state is a Gaussian distribution with different $\mu$ and $\sigma$. And observations are regarded as sample from the Gaussian distribution.
Additionally, the transition probability and the parameter of Gaussian distribution can be estimated through Baum-Welch algorithm. 

To test and apply the method, it estimates the volatility of securities. Yahoo Finance has provided the adjusted close price of TSLA and 
AMC from 2019-11-27 to 2022-11-23 (754 trading days). Because the returns series are more important for the problem, Table 1 show some
basic statistical information of the returns and hypothesis test. In the table, TSLA means the returns of TSLA and TSLA$^\star$ means 
the returns of the noise reduction data, noise means the noise removed from the original data. Additionally, K-S test is in order to 
discern whether the returns or noise is respect to normal distribution, and Ljung-Box test (L-B test) whose the maximum of lags is 12 
can be a white-noise test that makes sure the research is meaningful. To express the significance level, $^{***}$ is equal to $p<0.005$, 
$^{**}$ is equal to $p<0.01$, and $^*$ is equal to $p<0.05$. 

\begin{table}[htbp]
    \caption{Descriptive statistical analysis table}
    \centering
    \begin{tabular}{lccccll}
    \toprule
    \makecell[c]{Type} & Mean & Std & Skewness & Kurtosis & \makecell[c]{K-S test} & \makecell[c]{L-B test} \\
    \midrule
    TSlA & 0.003 & 0.045 & -0.260 & 3.306 & 0.066$^{***}$ & 16.323\\
    TSLA$^\star$ & 0.003 & 0.022 & -0.501 & 5.297 & 0.082$^{***}$ & 215.407$^{***}$\\
    noise & -0.005 & 6.900 & -0.111 & 0.108 & 0.019 & 1121.273$^{***}$\\
    \\
    AMC & -0.000 & 0.105 & 2.778 & 49.106 & 0.134$^{***}$ & 34.181$^{***}$\\
    AMC$^\star$ & 0.000 & 0.075 & 2.812 & 136.302 & 0.230$^{***}$ & 145.523$^{***}$\\
    noise & 0.011 & 1.694 & 0.485 & 2.117 & 0.073$^{***}$ & 1516.709$^{***}$ \\
    \bottomrule
    \end{tabular}
\end{table}

According to the descriptive statistical analysis, the noise reduction returns (TSLA$^\star$ and AMC$^\star$) have significant Spikes and 
thick tail, but aren't white-noise. It's worth noting that the rest noise still isn't a white-noise, which means there are more 
information included in the data. However, it isn't the essential of this research. 

Another indicator, conditional Volatility by GARCH, is usually used to estimate the volatility of securities. For comparison, GARCH(1,1)
fitted the returns of two stocks to estimate the volatility. Table 2 has shown the result of GARCH(1,1). It's easy to find the 
coefficients of the model basically have high significance level, so the model can fit the data very well and the conditional volatility
can be accepted. Additionally, GARCH(p,q) model with higher p and q can be also applied for estimating the volatility. But after testing,
the significance levels of coefficients have been out of the acceptable range. Thus, GARCH(1,1) is more advocated. 

\begin{table}[htbp]
\centering
\caption{GARCH Model Result} 
\begin{subtable}[t]{0.48\linewidth}
    \centering
    \caption{GARCH(1,1) for TSLA}
    \begin{tabular}{lccl}
        \toprule
                    & coef      & std err   & \makecell[c]{t}  \\
        \midrule
        $\mu$       & 2.34E-03  & 1.39E-03  & 1.68 \\
        $\omega$    & 3.44E-05  & 1.31E-05  & 2.63$^{**}$ \\
        $\alpha_1$  & 5.00E-02  & 2.07E-02  & 2.42$^*$ \\
        $\beta_1$   & 0.93      & 1.50E-02  & 61.98$^{***}$ \\
        \bottomrule
    \end{tabular}
\end{subtable}
\hspace{-0.5cm}
\begin{subtable}[t]{0.48\linewidth}
    \centering
    \caption{GARCH(1,1) for AMC}
    \begin{tabular}{lccl}
        \toprule
                    & coef      & std err   & \makecell[c]{t} \\
        \midrule
        $\mu$       &-6.75E-03  & 4.43E-03  & -1.52 \\
        $\omega$    & 1.07E-03  & 5.32E-04  & 2.01$^*$ \\
        $\alpha_1$  & 0.42      & 0.19      & 2.18$^*$ \\
        $\beta_1$   & 0.58      & 9.27E-02  & 6.28$^{***}$ \\
        \bottomrule
    \end{tabular}
\end{subtable}
\end{table}

The Gaussian HMM is set to have 3 potential states and 15 epochs to train the GaussianHMM for avoiding overfitting. The wavelet basis 
functions for denoising is Daubechies 8 (db8), which can guarantee the noise reduction data is smooth. The algorithm to estimate 
volatility need a series of sliding windows with different length. It sets the maximum of length to 80, and the minimum of length to 60. 
Then, TV volatility and GARCH volatility of TSLA and AMC can be calculated. Figure 7 and figure 8 show the details about these two kinds 
of indicators. 

\begin{figure}[htbp]
    \centering
    \includegraphics[width=\linewidth]{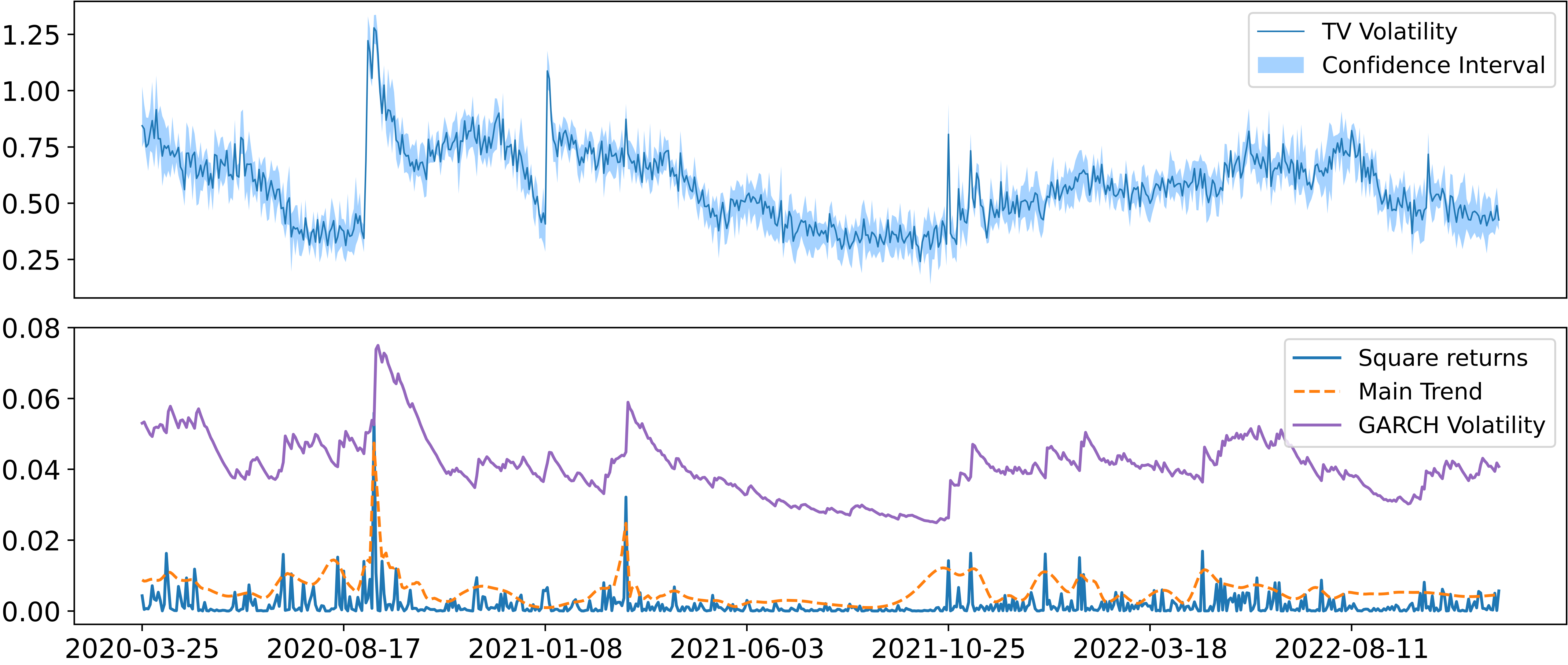}
    \caption{Comparison Chart about TV volatility and GARCH volatility of TSLA}
\end{figure}

\begin{figure}[htbp]
    \centering
    \includegraphics[width=\linewidth]{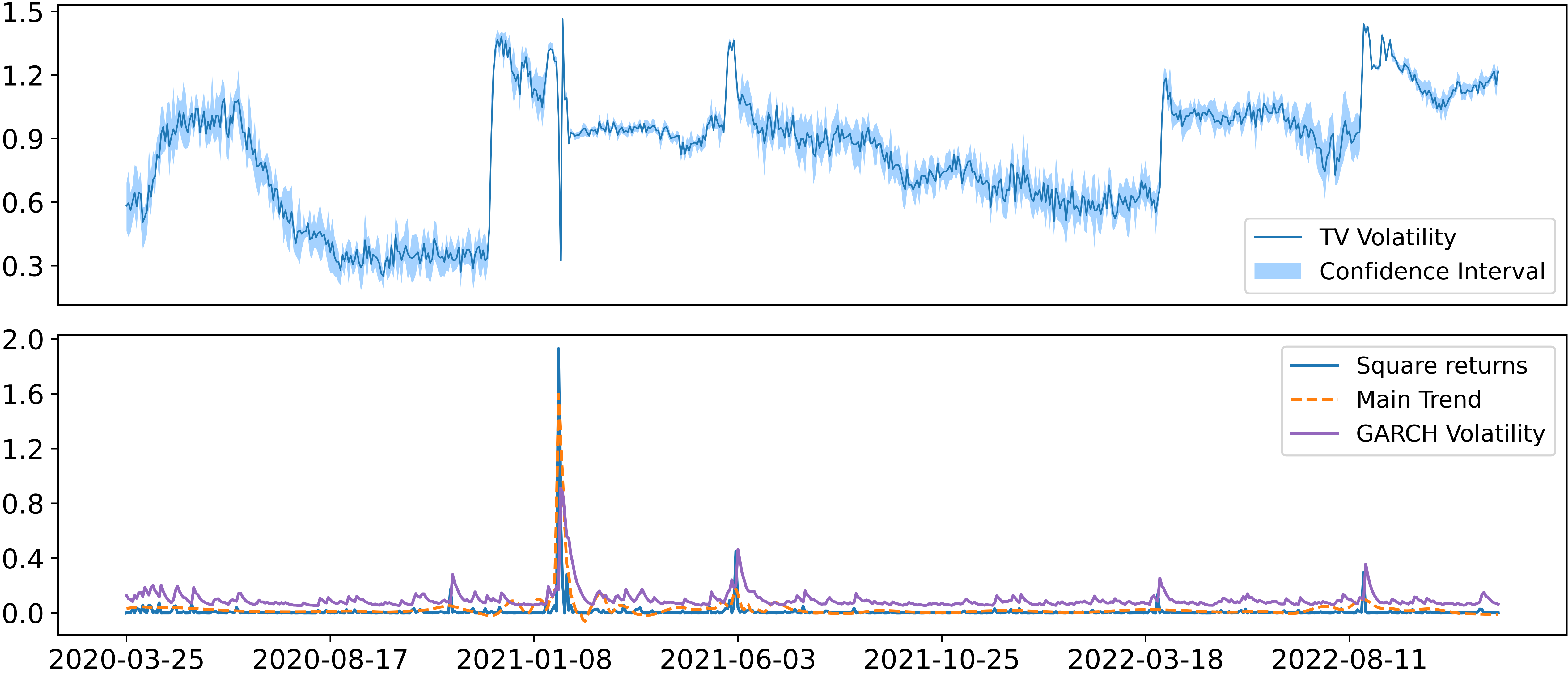}
    \caption{Comparison Chart about TV volatility and GARCH volatility of AMC}
\end{figure}

It's easy to find GARCH volatility is very close to the main trend of square returns, which means this indicator can reflect the 
instantaneous fluctuations of securities wonderfully. However, a huge shock from financial market can only effect the current volatility.
In practical experience, investors' confidence in the market will change dramatically after huge shock, and it always spends long time to
return to its previous stable level. During this period, the volatility (or risk) should also stay at high level. The short-sightedness 
of GARCH model maybe mislead investors possibly. Unfortunately, GARCH model with higher p and q will lead it's difficult to estimate the 
parameter, and it can't keenly catch the changes from instantaneous huge pulse. 

When the malfunction of traditional indicators is serious, TV volatility maybe provide another useful opinion to quantify the 
fluctuations of the security price. On one hand, TV volatility can reflect the influence caused by huge financial market shock, and the
change can last for a while, which is more convenient to discern volatility risk of financial market. On the other hand, the trend of 
TV volatility won't be too flat and smooth to reflect the changing trends of returns acutely. So, it can become a excellent indicator to 
measure the volatility of securities during a period.

\section{Conclusion}\label{sec5}
In this paper, the ergodicity and convergence bound of nonlinear Markov Chains is analyzed. According to some assumptions, nMC has been
transformed to normal linear Markov Chains. Then, the spectral radius of Coupling Markov transition probability matrix is applied to 
enhance the convergence condition. Finally, it gets the main result through optimization and triangle inequalities. There are two specific
examples proposed for testing and explaining this method. And it's effect to estimate convergence bounds has already surpassed the 
previous classic M-D condition. Moreover, it proposes a new indicator "TV Volatility" built on the condition. It's a excellent indicator
for investors to measure the risk and fluctuations of the securities in a period through the comparison between TV volatility and GARCH
volatility. 

As a very general model, Markov Chain has many other application scenarios besides quantitative finance. In these field, the analysis
about the convergence and ergodicity migth be included. So, the algorithm can provide a useful tool to do further research about epidemic
model, reinforcement and etc. Additionally, it's associated with time series data. It's also a valuable problem to investigate its utility
in other data modalities.

\bibliographystyle{unsrt}
\bibliography{article.bib}

\begin{appendices}
\section{The proof of Lemma 1}\label{appendixA}
Consider an arbitrary bounded measurable function $\mathscr{F}$. According to the definition (\ref{X1X2}), it always exists:
\begin{align*}
    \mathbb{E}(\mathscr{F}(\tilde{X}^1)) &=\mathbb{E}(\mathscr{F}(\tilde{X}^1)\mathbb{I}(\zeta=0))+
    \mathbb{E}(\mathscr{F}(\tilde{X}^1)\mathbb{I}(\zeta \neq 0)) \\
    & =\mathbb{E}(\mathscr{F}(\xi)\mathbb{I}(\zeta=0))+ \mathbb{E}(\mathscr{F}(\eta^1)\mathbb{I}(\zeta \neq 0)) \\
    & =\mathbb{E}(\mathscr{F}(\xi))\mathbb{E}(\mathbb{I}(\zeta=0))+\mathbb{E}(\mathscr{F}(\eta^1))\mathbb{E}(\mathbb{I}(\zeta \neq 0)) \\
    & =q \int \mathscr{F}(x)p^{\xi}(x)\varLambda(dx) + (1-q) \int \mathscr{F}(x)p^{\eta^1}(x)\varLambda(dx) \\
    & =q\int\mathscr{F}(x)\frac{p^1\wedge p^2}{q}(x)\varLambda(dx) + (1-q)\int\mathscr{F}(x)\frac{p^1-p^1\wedge p^2}{1-q}(x)\varLambda(dx)\\
    & =\int \mathscr{F}(x)p^1(x)\varLambda(dx)=\mathbb{E}(\mathscr{F}(X^1))
\end{align*}
On the other hand, the expection of $\mathscr{F}(\tilde{X}^1)$ can be expressed according to the definition.
\begin{equation*}
    \mathbb{E}(\mathscr{F}(\tilde{X}^1)) := \int \mathscr{F}(x)\tilde{p}^1(x)\varLambda(dx) = \int \mathscr{F}(x)p^1(x)\varLambda(dx)
\end{equation*}
Moreover, $Law(\tilde{X}^1)=Law(X^1)$ is equivalent to $\tilde{p}^1(x)\equiv p^1(x)$. Then, it try to provide the proof by contradiction.
It assume $\tilde{p}^1(x)\equiv p^1(x)$ is false, thus $\exists x_0$, $\tilde{p}^1(x_0)\neq p^1(x_0)$. Because $\mathscr{F}$ is arbitrary, 
it could be defined as:
\begin{equation*}
    \mathscr{F}(x) := 
    \begin{cases}
        1,\quad x_0-\varepsilon \leq x \leq x_0+\varepsilon \\
        0,\quad others
    \end{cases}
\end{equation*}
where $\varepsilon$ is small enough but isn't equal to 0. So, 
\begin{equation*}
    \int_{x_0-\varepsilon}^{x_0+\varepsilon}\tilde{p}^1(x)\varLambda(dx) = 
    \int_{x_0-\varepsilon}^{x_0+\varepsilon}p^1(x)\varLambda(dx)
\end{equation*}
By mean value theorem, $\exists \xi \in (x_0-\varepsilon,x_0+\varepsilon)$, $(\tilde{p}^1(\xi)-p^1(\xi))2\varepsilon=0$. It must have the 
equation $\tilde{p}^1(\xi)-p^1(\xi)=0$. When the $\varepsilon$ becomes smaller and smaller, $\xi$ infinitely approachs to $x_0$. Then, it 
obtains a contradiction and the assumption is wrong. Similarly, $\tilde{X}^2$ can be proved as the same method. So, the lemma \ref{lemma1} 
has been demonstrated. \qed

\section{The proof of Lemma 2}\label{appendixB} 
$\tilde{X}_n^1,\tilde{X}_n^2$ can share the distribution with $X_n^1,X_n^2$ according to lemma \ref{lemma1}. So, the prominent point to 
prove this lemma is to prove they have the same transition probability. Considering a state $i\in E$ (from the state space 
$\mathscr{S}(E,\mathscr{E})$), it always exists:
\begin{align*}
    P(\tilde{X}_{n+1}^1=i|\tilde{X}_n^1,\tilde{X}_n^2) &= 
    \mathbb{I}(\tilde{X}_n^1\neq\tilde{X}_n^2)P(\tilde{X}_{n+1}^1=i|\tilde{X}_n^1,\tilde{X}_n^2,\tilde{X}_n^1\neq\tilde{X}_n^2) \\
    &+ \mathbb{I}(\tilde{X}_n^1=\tilde{X}_n^2)P(\tilde{X}_{n+1}^1=i|\tilde{X}_n^1,\tilde{X}_n^2,\tilde{X}_n^1 = \tilde{X}_n^2)
\end{align*}

When $\tilde{X}_n^1\neq\tilde{X}_n^2$, it's easy to find $\zeta_n=1$, $\eta_n^1\neq\eta_n^2$ and $\xi_n$ can be set as any value. If 
$\tilde{X}_{n+1}^1=i$, there are two different situation that $\zeta_{n+1}=1,\eta_{n+1}^1=i$ or $\zeta_{n+1}=0,\xi_{n+1}=i$. Because of
the marginal transition probability definition, the equation can be written further.
\begin{align*}
    P(\tilde{X}_{n+1}^1=i|&\tilde{X}_n^1,\tilde{X}_n^2,\tilde{X}_n^1\neq\tilde{X}_n^2) =
    P(\tilde{X}_{n+1}^1=i|\tilde{X}_n^1=\eta_n^1,\tilde{X}_n^2=\eta_n^2,\eta_n^1\neq\eta_n^2) \\
    &=P^{\zeta}(1,1)P^{\eta^1}(\eta_n^1,i) + P^{\zeta}(1,0)P^{\xi}(\xi_n,i) \\
    &=(1-\kappa(\eta_n^1,\eta_n^2))\frac{P(\eta_n^1,i)-P(\eta_n^1,i)\wedge P(\eta_n^2,i)}{1-\kappa(\eta_n^1,\eta_n^2)} +
    \kappa(\eta_n^1,\eta_n^2)\frac{P(\eta_n^1,i)\wedge P(\eta_n^2,i)}{\kappa(\eta_n^1,\eta_n^2)} \\
    &=P(\eta_n^1,i)-P(\eta_n^1,i)\wedge P(\eta_n^2,i) + P(\eta_n^1,i)\wedge P(\eta_n^2,i) = P(\eta_n^1,i)
\end{align*}

When $\tilde{X}_n^1=\tilde{X}_n^2$, $(\tilde{X}_n^1,\tilde{X}_n^2)$ has reached a steady state. So, $\tilde{X}_n^1=\tilde{X}_n^2=\xi_n$ 
and $\zeta_n=0$. Thus, 
\begin{align*}
    P(\tilde{X}_{n+1}^1=i|\tilde{X}_n^1,\tilde{X}_n^2,\tilde{X}_n^1=\tilde{X}_n^2) &= 
    P(\tilde{X}_{n+1}^1=i|\tilde{X}_n^1,\tilde{X}_n^2,\tilde{X}_n^1=\tilde{X}_n^2=\xi_n) \\
    &=P^{\zeta}(0,1)P^{\eta^1}(\eta_n^1,i) + P^{\zeta}(0,0)P^{\xi}(\xi_n,i) = P^{\xi}(\xi_n,i)\mathbb{I}(\zeta_n=0) \\
    &=P(\xi_n,i) = P(\eta_n^1,i)
\end{align*}

Then, the transition probability can be expressed as:
\begin{align*}
    P(\tilde{X}_{n+1}^1=i|\tilde{X}_n^1,\tilde{X}_n^2)=\mathbb{I}(\tilde{X}_n^1\neq\tilde{X}_n^2)P(\eta_n^1,i) + 
    \mathbb{I}(\tilde{X}_n^1 = \tilde{X}_n^2)P(\eta_n^1,i)=P(\eta_n^1,i)
\end{align*}

It means that $\tilde{X}_n^1$ can be generated as $X_n^1$. Similarly, $\tilde{X}_n^2$ has the same property. So, the lemma \ref{lemma2}
has been demonstrated. \qed
\end{appendices}

\end{document}